\documentclass[12pt]{amsart}

\usepackage{ucs}

\usepackage{amssymb}
\usepackage{hyperref}
\usepackage{amsthm}
\usepackage{amsmath}
\usepackage{latexsym}
\usepackage[cp1251]{inputenc}
\usepackage{caption}
\usepackage{subcaption}
\usepackage{indentfirst}
\usepackage[left=2.6cm,right=2.6cm,top=2.8cm,bottom=2.8cm,bindingoffset=0cm]{geometry}
\usepackage{enumerate}

\usepackage[english]{babel}

\makeatletter 
\def\@seccntformat#1{\csname the#1\endcsname. } 
\def\@biblabel#1{#1.}

\makeatother

\title{On the prime spectrum of an automorphism group of an   ${\rm AT4}(p,p+2,r)$-graph}

\author{L. Yu. Tsiovkina} 
\address{N.N.~Krasovskii Institute of Mathematics and Mechanics of 
the Ural Branch of the  Russian Academy of Sciences, 
16 S.~Kovalevskaya str., Yekaterinburg, 620108 Russia}
\email{l.tsiovkina@gmail.com}
 
\date{}

\theoremstyle{plain}
\newtheorem{theorem}{Theorem}
\newtheorem{lemma}{Lemma}
\newtheorem{corollary}{Corollary}
\newtheorem{proposition}{Proposition}

\begin{document}

\vspace{\baselineskip}
\vspace{\baselineskip}

\vspace{\baselineskip}

\vspace{\baselineskip}

\begin{abstract}
This paper is devoted to the problem of classification of   ${\rm AT4}(p,p+2,r)$-graphs.
There is a unique ${\rm AT4}(p,p+2,r)$-graph with $p=2$, namely, the distance-transitive Soicher graph with intersection array $\{56, 45, 16, 1;1, 8,  45, 56\}$, whose local graphs are isomorphic to the Gewirtz graph.  
It is still unknown whether an  ${\rm AT4}(p,p+2,r)$-graph with $p>2$ exists. 
The local graphs of each ${\rm AT4}(p,p+2,r)$-graph are strongly regular with parameters $((p+2)(p^2+4p+2),p(p+3),p-2,p)$.
In the present paper, we find an upper bound for the prime spectrum of an automorphism group of a strongly regular graph with such parameters, and we also obtain some restrictions for the prime spectrum and the structure of 
an automorphism group of an ${\rm AT4}(p,p+2,r)$-graph in case, when  $p$ is a prime power.
As a corollary, we show that there are no arc-transitive  ${\rm AT4}(p,p+2,r)$\--graphs with  $p\in \{11,17,27\}$.
\end{abstract}

\maketitle

 \section*{Introduction}

Let  $\Gamma$ be a non-bipartite antipodal distance-regular graph of diameter 4 with eigenvalues $\theta_0>\theta_1>...>\theta_4$.  
In \cite{JKT}, J.~Koolen, A.~Juri\v{s}i\'{c} and P. Terwilliger proved that  parameters  $b_0,b_1$ and $a_1$
of $\Gamma$ satisfy the so-called {\it fundamental bound}:
$$(\theta_1+\frac{b_0}{a_1+1})(\theta_4+\frac{b_0}{a_1+1})\ge -\frac{b_0a_1b_1}{(a_1+1)^2}.$$
If it holds with equality, then the local  graphs of $\Gamma$ are strongly regular, 
and the intersection array
of $\Gamma$  is expressed in terms of the non-principal eigenvalues $p= -1-b_1/(1+\theta_4)$ and $-q=-1-b_1/(1+\theta_1)$ of the local graphs  and
the antipodality index $r$ of $\Gamma$. In this case, $\Gamma$  is called an {\it antipodal tight graph} of diameter 4 with parameters  $(p,q,r)$ or simply an {\it ${\rm AT4}(p,q,r)$-graph}.

 The study of the class of ${\rm AT4}(p,q,r)$-graphs is motivated by the fact that it contains 11 of all known 14
non-bipartite antipodal distance-regular graphs of diameter $4$, including the famous antipodal triple cover 
of the 3-transposition graph of the sporadic Fischer group ${Fi_{24}}'$  (see \cite{JKT}). 
Most  examples   are related to some important permutation representations 
of various quasisimple groups or have pseudo-geometric local graphs.

In this paper, we consider the problem of classification of  ${\rm AT4}(p,p+2,r)$-graphs. 
It is known that the local graphs of an ${\rm AT4}(p,p+2,r)$-graph are strongly regular with parameteres 
$$((p+2)(p^2+4p+2),p(p+3),p-2,p).$$ Besides,   by a theorem of Juri\v{s}i\'{c} (see \cite[Theorem 5.5]{Jur}), 
each second subconstituent  of an ${\rm AT4}(p,p+2,r)$-graph is again a non-bipartite antipodal distance-regular
 graph of diameter $4$. Thus, the family of ${\rm AT4}(p,p+2,r)$-graphs is a potentially rich source of
new examples of non-bipartite antipodal distance-regular graphs of diameter $4$.

 By a Brouwer's result (see \cite[Theorem 11.4.6]{BCNcorr}), there is a  unique 
 ${\rm AT4}(p,p+2,r)$-graph with  $p=2$, namely, it is the first Soicher graph
  (with intersection array $\{56, 45, 16, 1;1, 8,  45,\\
   56\}$). 
In \cite{Soi} it is shown that this   graph admits an arc-transitive group of automorphisms, 
which is isomorphic to the quasisimple group $3_2.U_4(3)$ (in the  notation of~\cite{atlas}).
To date,   no examples of  ${\rm AT4}(p,p+2,r)$-graphs  with  $p>2$ are known. 
A natural  task in classification of  ${\rm AT4}(p,p+2,r)$-graphs is to describe those   
ones, which possess arc-transitive groups of automorphisms (and, therefore, can be constructed 
as orbital graphs of some permutation groups). One approach to the problem of classification of 
arc-transitive ${\rm AT4}(p,p+2,r)$-graphs involves a necessity to determine the prime spectrum
of an automorphism group of such a graph. 

In recent papers \cite{Tsi19, Tsi18b}, it was noticed that, for $p\in\{5,7\}$, the prime spectrum of an  
automorphism group of an ${\rm AT4}(p,p+2,r)$-graph
  depends essentially on the structure    of its local graphs. 
Here we generalize this observation to the case, when   $p$ is  a prime power.
 
\begin{theorem}\label{T1} Let $\Gamma$ be an ${\rm AT4}(p,p+2,r)$-graph, let  $\{a,b\}$ 
be an edge of $\Gamma$ and $G={\rm Aut}(\Gamma)$. Then $G_a$ acts faithfully both on  $\Gamma_1(a)$ and on $\Gamma_2(a)$, and if  $p$  is  a prime power, $p>2$,   then  $\pi(G_{a,b})\subseteq \{2,3,...,p\}$.
\end{theorem}
\begin{corollary}\label{C1}Suppose that $\Gamma$ is an  ${\rm AT4}(p,p+2,r)$-graph, where $p$ is a prime  power, $p>2$, $a\in \Gamma$ and the group $G={\rm Aut}(\Gamma)$ acts transitively on arcs of $\Gamma$. Then
\begin{equation*}
\pi((p+2)(p^2+4p+2)(p+1)(p+4))\subseteq\pi(G)\newline
\subseteq \{2,3,...,p+2\}\cup \pi((p^2+4p+2)(p+4)).
\end{equation*}
In addition, if the numbers $p+2$ and $p^2+4p+2$ are prime, then the group $G_a$ is almost simple.
\end{corollary}

In \cite{Tsi19,MN2}, it was shown that there are no arc-transitive ${\rm AT4}(p,p+2,r)$-graphs  with $p\in \{3,5\}$. 
Corollary~\ref{C1} allows us to obtain a similar result in case   $p\in \{11,17,27\}$.
By determining the structure of admissible   groups of automorphisms of ${\rm AT4}(p,p+2,r)$-graphs 
with $p$ being a prime power such that numbers $p+2$ and $p^2+4p+2$ are primes not greater than 1000,  
we prove the next theorem.
\begin{theorem}\label{T2}   ${\rm AT4}(p,p+2,r)$-graphs with $p\in \{11,17,27\}$ are not arc-transitive.
\end{theorem}

Next we list  some main terminology and  notation that are used in this paper.  

In the paper we consider undirected graphs without loops or multiple edges.    
For a vertex $a$ of a graph  $\Gamma$ by $\Gamma_i(a)$ we denote  {\it $i$-neighborhood} of $a$, 
that is, the subgraph of $\Gamma$ which is induced by the set of  vertices that are at distance  $i$ from $a$.
By a {\it local graph} of a graph $\Gamma$ we mean a  $1$-neighborhood of some vertex of $\Gamma$. 
A connected graph  $\Gamma$ of diameter $d$ is called  {\it distance-regular}, if there are constants
 $c_i,a_i$ and $b_i$ such that for all $i\in \{0,1,\ldots,d\}$ and each pair  of vertices $u$ and $ w$ which are at distance  $i$ in $\Gamma$, the following equalities hold: $c_i=|\Gamma_{i-1} (u)\cap \Gamma_1(w)|$,  $a_i=|\Gamma_{i} (u)\cap \Gamma_1(w)|$
  and  $b_i=|\Gamma_{i+1} (u)\cap \Gamma_1(w)|$ (it is assumed that $b_d=c_0=0$), and, in particular, $|\Gamma_1(u)|=b_0=c_i+a_i+b_i$.  The  sequence $\{b_0, b_1,\ldots, b_{d-1} ; c_1,\ldots, c_d \}$ is called the {\it 
  intersection array} of a distance-regular graph. Further, if the binary relation ``to be at distance 0 or d'' on the set of vertices of a connected graph $\Gamma$ of diameter $d$  is an equivalence relation, 
  then the graph $\Gamma$ is called {\it antipodal} and  classes of this relation are called {\it antipodal classes} of $\Gamma$. It is well-known that distance-regular graphs split into the families of primitive graphs and imprimitive graphs, and by a theorem of Smith, each imprimitive distance-regular graph of valency at least 3 is bipartite or antipodal.

For a graph  $\Gamma$ and a subset $X\subseteq {\rm Aut}(\Gamma)$ by 
${\rm Fix}(X)$ we denote the set of all vertices of $\Gamma$, that are fixed by every element of $X$, 
which we will also identify with a subgraph of $\Gamma$ induced by ${\rm Fix}(X)$.  
A graph is called {\it arc-transitive}, if its automorphism group acts transitively on the set of its arcs
(ordered pairs of adjacent vertices).
Further, for a natural number $n$ by $\pi(n)$ we denote the set of its prime divisors.
 For a finite group $G$  the set  $\pi(|G|)$ is called  the {\it prime spectrum} of $G$, which is shortly denoted by  $\pi(G)$.

The parameters $b_0, a_1$ and $c_2$ of a distance-regular graph will be also written as   $k,\lambda$ and $\mu$,
respectively.
A distance-regular graph of diameter 2 on  $v$ vertices is called  a {\it strongly regular graph}  with parameters $(v,k,\lambda,\mu)$. 
If one of the parameters $k,\lambda$ and $\mu$ is defined in an arbitrary graph $\Omega$,
then this parameter of $\Omega$ will be also written as  $k_{\Omega}$, $\lambda_{\Omega}$ or $\mu_{\Omega}$, respectively. 
 
Let  $\Gamma$ be an antipodal graph. By  ${\mathcal F}(\Gamma)$ we will denote the set of antipodal classes of $\Gamma$. 
The graph, denoted by $\bar\Gamma$,   with the vertex set ${\mathcal F}(\Gamma)$, 
in which two vertices $F_1$ and $F_2$ are adjacent if and only if $\Gamma_1(a)\cap F_2\neq \varnothing$ for some vertex $a\in F_1$, is called an {\it antipodal quotient} of $\Gamma$. The graph $\Gamma$ is called an
 {\it antipodal  $r$-cover} of $\bar\Gamma$, if, for every two antipodal classes $F_1,F_2\in {\mathcal F}(\Gamma)$, 
 we have $r=|F_1|=|F_2|$ and the subgraph of  $\Gamma$ that is induced by $F_1\cup F_2$ is a
 perfect matching or a coclique.   
 
Let $\Gamma$ be a distance-regular graph of diameter 4 with intersection array $\{b_0,b_1,b_2,b_3;c_1,c_2,\\ 
c_3,c_4\}.$
By~\cite[Proposition 4.2.2]{BCN} the graph $\Gamma$ is antipodal if and only if $b_i=c_{4-i}$
for all $i\in\{0,1,3,4\}$.
In this case,  $\Gamma$ is an antipodal  $r$-cover of  $\bar{\Gamma}$,  $r=1+b_2/c_2$ and the graph  $\bar{\Gamma}$   is strongly regular with parameters $(v/r,b_0,a_1,rc_2)$, where $v=r(b_0+1)+b_0b_1/c_2$ is the number 
of vertices in~$\Gamma$.
\medskip

In order to prove Theorem~\ref{T1}, we will use the following basic properties of ${\rm AT4}(p,p+2,r)$-graphs (see  \cite{Jur}, \cite{JK07}
and  \cite{GMP13}; we also note here a misprint    
in the parameters $b_0$ and  $c_4$ of an ${\rm AT4}(p,p+2,r)$-graph in  \cite[Theorem 4]{GMP13}: there should be written ``$(p+2)(p^2+4p+2)$'' instead of ``$(p+1)(p+2)^2$''). 
 
\begin{proposition}\label{P1}
Let $\Gamma$ be an ${\rm AT4}(p,p+2,r)$-graph and let  $\bar \Gamma$ be its antipodal quotient. 
Then the following statements hold:

  $(1)$ $2p(p+1)(p+2)/r$ is even, $2<r<p+2$, $r$  divides $2(p+1)$ and $\Gamma$ has intersection array  \begin{multline*}
  \{(p+2)(p^2+4p+2),(p+3)(p+1)^2,(r-1)2(p+1)(p+2)/r,1;\\
  1,2(p+1)(p+2)/r,(p+3)(p+1)^2,(p+2)(p^2+4p+2)\};
  \end{multline*}

  $(2)$  $\bar \Gamma$ is a strongly regular graph with parameters  
  $$((p+1)^2(p+4)^2/2,(p+2)(p^2+4p+2),p(p+3),2(p+1)(p+2))$$ and non-principal eigenvalues 
  $p,-(p^2+4p+4)$;

  $(3)$ the second neighborhood of each vertex in  $\bar \Gamma$   is a strongly regular graph with parameters   $$((p+1)(p+3)(p^2+4p+2)/2,p(p+2)^2,p^2+p-2,2p(p+1))$$ and non-principal eigenvalues 
  $p,-(p^2+2p+2)$;
  
  $(4)$ the second neighborhood of each vertex in $\Gamma$ is a distance-regular graph with intersection array  
  $$\{p(p+2)^2,(p+1)^3,2(r-1)p(p+1)/r,1;1,2p(p+1)/r,(p+1)^3,p(p+2)^2\}.$$ 
\end{proposition}
\medskip 
 
 This paper is arranged as follows. 
In Section~\ref{Sec1} we investigate   automorphism groups of strongly regular graphs with parameters $((p+2)(p^2+4p+2),p(p+3),p-2,p)$.
In Section~\ref{Sec2} we obtain several restrictions on the prime spectrum of an automorphism group of an ${\rm AT}4(p,p+2,r)$-graph, and prove Theorems~\ref{T1}, \ref{T2} and Corollary~\ref{C1}.

 \section{Automorphisms of strongly regular graphs with parameters  $((p+2)(p^2+4p+2),p(p+3),p-2,p)$}\label{Sec1}

In this section    $\Theta$  is a strongly regular graph with parameters $$((p+2)(p^2+4p+2),p(p+3),p-2,p),$$ $\widetilde{G}={\rm Aut}(\Theta)$ and $v=(p+2)(p^2+4p+2)=(p+2)((p+2)^2-2)$. 
Note that $\Theta$ has spectrum $$p(p+3)^1,  p^{(p+3)((p+2)^2-2)/2},  (-p-2)^{(p+1)((p+2)^2-2)/2-1}$$ and by~\cite[Proposition 1.3.2]{BCN}
the size of a clique in $\Theta$ does not exceed $(p+2)^2$.

Denote by $\psi$ the matrix representation of $\widetilde{G}$ in
$GL_{v}(\mathbb{C})$, induced by permutation representation of $\widetilde{G}$ on the vertices of $\Theta$, 
and let $\alpha_j(g)$ be the number of vertices $x$ of $\Theta$ such that $d(x,x^g)=j$. 
The linear  space  ${\mathbb{C}}^{v}$ is an orthogonal direct sum of  $\psi(\widetilde{G})$-invariant
eigenspaces of the adjacency matrix of $\Theta$  (see~\cite[\S\ 3]{Cam99}). 
Let us calculate the formulas for the characters of projections of $\psi$ on subspaces with dimensions $n_1={(p+3)((p+2)^2-2)/2}$ and $n_2=(p+1)((p+2)^2-2)/2-1$.

\begin{lemma}\label{L1}  
If $g\in \widetilde{G}$ and $\chi_i$ is the character of the projection of $\psi$ 
on the subspace with dimension $n_i$, where $i\in \{1,2\}$, 
 then  
\begin{multline*}\chi_1(g)=\frac{(p+3)\alpha_0(g)/2+\alpha_1(g)/2-\alpha_2(g)/(2(p+1))}{p+2}  \textit{ and }\\
\chi_2(g)=\frac{p(p+3)\alpha_0(g)/2-(p+2)\alpha_1(g)/2+p\alpha_2(g)/(2(p+1))}{(p+2)^2-2}.
\end{multline*}
If $|g|$ is a prime, then $\chi_1(g)-n_1$ and  $\chi_2(g)-n_2$ are divisible by $|g|$. 
\end{lemma}\begin{proof}
Due to results in~\cite[\S\ 3.7]{Cam99} it suffices to note that the second eigenmatrix 
$Q$ of $\Theta$ equals 
$$\left (
 \begin{array}{ccc}
         1 & 1 & 1 \\
         (p+3)((p+2)^2-2)/2 &  (p+2)^2/2-1 & -((p+2)^2-2)/(2(p+1))  \\
         (p+1)((p+2)^2-2)/2-1 & -(p+2)^2/2 & p(p+2)/(2(p+1))
 \end{array}
 \right ).$$
 
The lemma is proved.
\end{proof}
Note that the values of $\chi_i$, which are calculated by the formulas from Lemma~\ref{L1}, are integer.

\begin{lemma}\label{L2} If   $\Theta$ contains a proper strongly regular subgraph $\Delta$ with parameters  $(v',k',p-2,p)$, then $k'-p+1$ is a square, $p$ divides $(k'-t')(k'-s')$, and  $2p$  divides 
$k'(k'-s')$, where $t'$ and $s'$ are non-principal eigenvalues of $\Delta$, $s'<0$, and if moreover
$p$ is a prime power, the one of the following statements holds:
\begin{itemize}
\item[$(1)$]  $p=t'+1=-(s'+1), k'=p^2+p-1, v'=p^2(p+2)$;
\item[$(2)$]  $s'=-p$, $k'=p(p-1), v'=p((p-1)^2+1)$;
\item[$(3)$]  $s'=-p/2$, $k'=p^2/4,  v'=(p^2/8-p/4+1)(p/2+1)$ or $s'=-3, p=2$;
\item[$(4)$]  $t'=p/2$, $k'=p(p/2+4)/2, v'=(3+p/2)(p^2/8+5p/4+1)$.
 
\end{itemize}
\end{lemma}
\begin{proof} Suppose that $\Delta$ is a strongly regular subgraph of $\Theta$
with parameters $(v',k',p-2,p)$ and  $a\in \Delta$. Then by \cite[Theorem 1.3.1]{BCN} $D^2=k'-p+1<(p+1)^2, t'=-1+\sqrt{1-p+k'}$ has multiplicity $k'(k'-s')/2p$, $ s'=-(t'+2)$ and $v'=(k'-t')(k'-s')/p$.

Let further $p$ be a prime power.  Since $|\Delta_2(a)|=k'(k'-p+1)/p$ and the number $(k',k'-p+1)$ divides $p-1$, 
then either $p$ divides $k'$, or $p$ divides $k'+1$.

If $p$ divides $k'+1$, then $p$ does not divide $s'(s'+2)$, hence 
 $p=t'+1=-(s'+1), k'=p^2+p-1, v'=p^2(p+2)$.
 
Let  $p$ divide $k'$. Then $p$ divide $s'(s'+2)$. 
Since $t'<p,$ then either $s'=-p$, $k'=p(p-1), v'=p((p-1)^2+1)$, or $s'=-p/2$, $k'=p^2/4,
  v'=(p^2/8-p/4+1)(p/2+1)$, or $s'=-3p/2, p=2$,  or $t'=p/2$, $k'=p(p/2+4)/2, v'=(3+p/2)(p^2/8+5p/4+1)$. 
The lemma is proved.  \end{proof}

\begin{lemma}\label{L3} Let  $g\in \widetilde{G}$ be an element of prime order and  $\Omega={\rm Fix}(g)$. 
Then \begin{multline*}
\alpha_1(g)=2(p+1)(|g|z_1+{(p+3)((p+2)^2-2)/2})-(p+2)|\Omega|+((p+2)^2-2)=\\
p|\Omega|-2(p+1)({(p+1)((p+2)^2-2)/2}-1+|g|z_2)+p(p+2)
\end{multline*}
 for some $z_1, z_2\in \mathbb{Z}$ and the following statements hold.
 
\begin{itemize}
\item[$(1)$] If $\Omega\neq \varnothing$, $p$ is a prime power and $p>2$,  then $|\Omega|\le (p+2)^2-2$ and either 
\begin{itemize}
\item[$(i)$] $|g|<p$, or
\item[$(ii)$]  $|g| = p$, $ |\Omega|\equiv 4 \pmod p$,  $\alpha_1(g)=p(|\Omega|-2(p+1)z+p+2)$
for some $z\in \mathbb{Z}$, and each connected component $\Delta$ of $\Omega$ is a single vertex or an amply regular graph of diameter at least  $3$ with parameters 
$(|\Delta|, pl,p-2,p)$, where $\Delta\ge 4(p+1)$ and $l\in \{2,4,...,p+2\}$.

\end{itemize}
\item[$(2)$] If $\Omega= \varnothing$, then either
\begin{itemize}
\item[$(i)$] $|g|$ is odd and divides $(p+2)^2-2$,
$\alpha_1(g)=2(p+1)|g|z_1+(p+2)^2-2$
 for some $z_1\in \mathbb{Z}$, in particular,   $\alpha_1(g)=|g|$ if $|g|=(p+2)^2-2$, or
\item[$(ii)$] $|g|$ is odd and divides $(p+2)$, $\alpha_1(g)=-2(p+1)|g|z_2+p(p+2)$
for some $z_2\in \mathbb{Z}$, or
\item[$(iii)$] $|g|=2$  and  $p$ is even, $\alpha_1(g)=-4(p+1)z_3+p(p+2)$
for some $z_3\in \mathbb{Z}$.
 \end{itemize}
  \end{itemize}
\end{lemma}
\begin{proof} 
Let $a\in \Omega\neq \varnothing$. Then by \cite[Theorem 3.2]{BehLam} $|\Omega|\le p\cdot v/(k-p)=(p+2)^2-2$. We have  $$x_1=|\Omega\cap \Theta_1(a)|\equiv p(p+3)\pmod {|g|},$$  $$x_2=|\Omega\cap \Theta_2(a)|\equiv (p+3)(p+1)^2\pmod {|g|},$$  $$|\Theta-\Omega|=v-1-x_1-x_2\equiv 0\pmod {|g|},$$
in particular, if $|g|$ does not divide $p(p+3)(p+1)$, then $x_1x_2>0$.

Let $|g|>p$.  Then for each vertex $b\in \Omega-\{a\}$ we have 
$\Theta_1(a)\cap \Theta_1(b) \subset \Omega$ and by  \cite[Proposition 1.1.2]{BCN}, $\Omega$ 
is a strongly regular graph with parameters 
$(|\Omega|, x_1, p-2,p)$ and $|\Omega|=1+x_1+x_1(x_1-p+1)/p$.
Let $k'=x_1$ and $v'=|\Omega|$.
By Lemma~\ref{L2} it suffices to consider the following four cases.

\quad{$(1)$} If $k'=p^2+p-1$ and $v'=p^2(p+2)$, then $v'> (p+2)^2-2,$ a contradiction.

\quad{$(2)$} If $k'=p(p-1)$ and $v'=p((p-1)^2+1)$, then $v'> (p+2)^2-2,$ a contradiction.

\quad{$(3)$} Let $k'=p^2/4$ and $v'=(p^2/8-p/4+1)(p/2+1)$. Since $v'\le (p+2)^2-2,$
 we get $p\in \{4,8,16\}$.
But $p^2/4\equiv p(p+3)\pmod {|g|}$ and $|g|\le p(p+3)$, a contradiction for all admissible pairs $(p,|g|)$.

\quad{$(4)$} If $k'=p(p/2+4)/2$ and $v'=(3+p/2)(p^2/8+5p/4+1)$, then $v'>(p+2)^2-2$, again a contradiction.

Let  $p$ be a prime and $|g|=p$. Then $\lambda_{\Omega}=p-2$ and for each vertex  $b\in \Omega-\{a\}$ 
we have
$|\Theta_1(a)\cap \Theta_1(b) \cap \Omega|\in \{0,p-2,p\}$.
Let $\Delta$ be a connected component of $\Omega$ and $|\Delta|>1$. 

If $d(\Delta)=1$, then $\Delta\simeq K_{p}$, but $|\Omega(a)|\equiv 0\pmod p$ for $a\in\Delta$, a contradiction.

Suppose that $d(\Delta)>1$. Then  $\lambda_{\Delta}=p-2$ and  $\mu_{\Delta}=p$
  and by  \cite[Proposition 1.1.2]{BCN}, $\Delta$ is an amply regular graph   with parameters 
$(|\Delta|, k_{\Delta},p-2,p)$ and $|\Delta|\ge 1+ k_{\Delta}+k_{\Delta}(k_{\Delta}-p+1)/p$.
If  $d(\Delta)=2$, then $|\Delta|= 1+ k_{\Delta}+k_{\Delta}(k_{\Delta}-p+1)/p$ and $k_{\Delta}=pl$, where $l\ge 1$.
By Lemma~\ref{L2}    $k_{\Delta}=p(p-1)$ and $|\Delta|=p((p-1)^2+1)$, which is impossible as $|\Delta|\le (p+2)^2-2$.

Therefore, $d(\Delta)\ge 3$,  and by  \cite[Theorem 1.5.5]{BCN}, $k_{\Delta}(k_{\Delta}-p+1)/p\ge k_{\Delta}\ge 2p-1$.
As above, we get   $k_{\Delta}=pl,$ where $l\ge 2$, $|\Omega|\equiv 4\pmod p$, hence $|\Delta|\ge 4(p+1)$ and $|\Omega-\Delta|\le p^2-2$.

Now let $|g|$ and $\Omega$ be arbitrary. By Lemma~\ref{L1} 
$$
|g|z=\frac{(p+3)\alpha_0(g)/2+\alpha_1(g)/2-\alpha_2(g)/(2(p+1))}{p+2}-{(p+3)((p+2)^2-2)/2}$$
 for some $z\in \mathbb{Z}$. 
Hence \begin{multline*}(p+1)(p+3)|\Omega|+(p+1)\alpha_1(g)-\alpha_2(g)= 
2(p+1)(p+2)(|g|z+{(p+3)((p+2)^2-2)/2}).
\end{multline*} 
Taking into account that $(p+2)((p+2)^2-2)-|\Omega|-\alpha_1(g)=\alpha_2(g)$, we obtain
\begin{multline*}
2(p+1)(p+2)(|g|z+{(p+3)((p+2)^2-2)/2})=\\(p+1)(p+3)|\Omega|+(p+1)\alpha_1(g)-
((p+2)((p+2)^2-2)-|\Omega|-\alpha_1(g))= \\
(p^2+4p+4)|\Omega|+(p+2)\alpha_1(g)-(p+2)((p+2)^2-2)
\end{multline*} and  $\alpha_1(g)=2(p+1)(|g|z+{(p+3)((p+2)^2-2)/2})-(p+2)|\Omega|+((p+2)^2-2)$.

Again by Lemma~\ref{L1}
\begin{multline*}
|g|z=\frac{p(p+3)\alpha_0(g)/2-(p+2)\alpha_1(g)/2+p\alpha_2(g)/(2(p+1))}{(p+2)^2-2}-{(p+1)((p+2)^2-2)/2}+1
\end{multline*} for some $z\in \mathbb{Z}$. 
Hence \begin{multline*}(p+1)p(p+3)|\Omega|-(p+1)(p+2)\alpha_1(g)+p\alpha_2(g)=\\
2(p+1)({(p+2)^2-2})({(p+1)((p+2)^2-2)/2}-1+|g|z).
\end{multline*}  
As above,
\begin{multline*}
2(p+1)({(p+2)^2-2})({(p+1)((p+2)^2-2)/2}-1+|g|z)=\\
(p+1)p(p+3)|\Omega|-(p+1)(p+2)\alpha_1(g)+p((p+2)((p+2)^2-2)-|\Omega|-\alpha_1(g))=\\
p(p^2+4p+2)|\Omega|-(p^2+4p+2)\alpha_1(g)+p(p+2)((p+2)^2-2)
\end{multline*} and  $\alpha_1(g)=p|\Omega|-2(p+1)({(p+1)((p+2)^2-2)/2}-1+|g|z)+p(p+2)$.

The lemma is proved.
\end{proof}

  \begin{lemma}\label{L4} If a subgroup $G\le \tilde{G}$ is transitive on the vertices of $\Theta$,
  then for each vertex $a\in \Theta$ the set  ${\rm Fix}(G_a)$ is an imprimitivity block of $G$, 
  the group $N_G(G_a)$ acts transitively on ${\rm Fix}(G_a)$ and $|{\rm Fix}(G_a)|=|N_G(G_a):G_a|$. 
  If moreover  $G_a\ne 1$, then the number $|{\rm Fix}(G_a)|$ does not exceed  $(p+2)^2-2$ and divides $(p+2)((p+2)^2-2)$.   
 \end{lemma}
 \begin{proof}
As it was noted in the proof of the previous lemma,  by \cite[Theorem 3.2]{BehLam} the order of every subgraph of fixed points of a non-trivial automorphism of $\Theta$ is not greater than $(p+2)^2-2$.
The remaining claims are just a corollary of well-known facts (eg., see \cite[Exercise 1.6.3]{DM}).
The lemma is proved. \end{proof}
\medskip
 
In Lemmas~\ref{L5} and \ref{L6} it is assumed that $p$ is a  prime power and $p>2$.
\begin{lemma}\label{L5} Let $(p+2)^2-2\in \pi(\widetilde{G})$, $g$ be an element of order $(p+2)^2-2$ of  $\tilde{G}$ and let   $f$ be an element of prime order of $C_{\tilde{G}}(g)$.
Then either $f\in \langle g\rangle$, or  $|f|<p, |f|$ divides $p+1$, ${\rm Fix}(f)$ is a regular graph on $(p+2)^2-2$ vertices, $\alpha_1(f)=(p+1)((p+2)^2-2)$ and every non-single point $\langle f\rangle$-orbit is a clique. 
 \end{lemma}
\begin{proof}  
Put $s=(p+2)^2-2$. 
As  $|{\rm Fix}(f)|$ is divisible by $s$, it follows by Lemma~\ref{L3} that   $|{\rm Fix}(f)|\in \{s,0\}$ and $$\alpha_1(f)=2(p+1)(|f|z_1+{(p+3)s/2})-(p+2)|{\rm Fix}(f)|+s$$ is divisible by $s$, hence $|f|z_1$ 
is divisible by $s$.
 
Suppose that ${\rm Fix}(f)=\varnothing$. If $|f|=s$ and $f\notin \langle g\rangle$, then by
Lemma~\ref{L3} the group $\langle f,g\rangle$ 
 acts semiregularly on the vertices of  $\Theta$, which contradicts to the fact  $s^2$ does not divide $(p+2)s$.
Suppose that $f\notin \langle g\rangle$.
 By Lemma~\ref{L3} $|f|$ divides $p+2$,  hence $z_1$ is divisible by $s$ and
 $$ \alpha_1(f)=2(p+1)s(|f|z_1/s+{(p+3)/2})+s\le v.$$ In case when $\alpha_1(f)\le (p+1)s$ 
 we get 
 $\alpha_1(f)=s$ and ${(p+3)/2}=-|f|z_1/s,$ but $(|f|,p+3)=1$, a contradiction.
 Hence $\alpha_1(f)=v$, every   $\langle f\rangle$-orbit is a clique and  $|f|\le p$.
 But by Lemma~\ref{L3}  $\alpha_1(g)=s$ is divisible by $|f|$, a contradiction.

Now let $|{\rm Fix}(f)|=s$. By Lemma~\ref{L3} $|f|\le p$,  hence $z_1$ is divisible by $s$
and 
 $$\alpha_1(f)=(p+1)s(2|f|z_1/s+p+3)-(p+1)s=(p+1)s(2|f|z_1/s+p+2)\le v-s
 .$$ Hence either $\alpha_1(f)=0$ and $2|f|z_1/s+p+2=0$, or $p+1= -2|f|z_1/s.$ 
Note that $|f|$ does not divide $p+2$  (otherwise the number $p+2$ would be even, which is impossible since by the assumption $s$ is a prime).
Hence  $|f|$ divides $p+1$, $\alpha_1(f)=v-s=(p+1)s$ and every non-single point  $\langle f\rangle$-orbit
is a clique.    

The lemma is proved. 
\end{proof}

  \begin{lemma}\label{L6}  Let $a\in \Theta$ and suppose $G$ is a subgroup of $\tilde{G}$ that is transitive on the vertices of  $\Theta$  and $s=(p+2)^2-2\in \pi({G})$. Then the following statements hold.
  \begin{itemize}
\item[$(1)$] If $G_a\ne 1$, then either $|{\rm Fix}(G_a)|=s$ and $\pi(G_a)\subseteq \pi(p+1)$,
or  $|{\rm Fix}(G_a)|$ divides $p+2$.  

 \item[$(2)$]  If the group $G$ is solvable,
then either \begin{itemize}
  \item[$(i)$] $O_s(G)\ne 1$, $p+2$ is a power of $3$, $s\equiv 1 \pmod 3$ and  $\pi(G_a)\subseteq \pi((p+1)(p^2+4p+1))$, 
  or \item[$(ii)$] $O_s(G)=1$, $G$ contains a minimal normal subgroup of order $t^e$, where $t\in \pi(p+2)$ and $e\ge 2$,
$t^e\equiv 1\pmod s, $ $s$ divides $|GL_e(t)|$ and the number   $p+2$ is composite.
  \end{itemize}
\end{itemize} 
\end{lemma}
 \begin{proof}
Put $s=(p+2)^2-2$ and $X=G_a$ for a vertex $a\in \Theta$. 
Suppose  $G$ is a subgroup of  $\tilde{G}$ that is transitive on the vertices of  $\Theta$
and  $g$ is an element of order  $s$ of ${G}$.
 
Let $X\ne 1$. Then by Lemma~\ref{L5}  $\pi(C_X(g))\subseteq \pi(p+1)$ and by Lemma~\ref{L4} the number  
 $|{\rm Fix}(X)|$ equals  $s$ or divides $p+2$.
If $|{\rm Fix}(X)|=s$, then ${\rm Fix}(X)$ is  $\langle g\rangle$-invariant,
that is the group $\langle g\rangle$ acts regularly on the vertices of the graph ${\rm Fix}(X)$ and normalizes $X$.
In this case for each element $f$ of prime order of  $X$ we have ${\rm Fix}(X)={\rm Fix}(f)$
and hence $s\equiv s(p+2)\pmod{|f|},$ thus $|f|$ divides $p+1$ and  $\pi(X)\subseteq \pi(p+1)$.
The statement (1) is proved. 

Let us prove statement (2). Suppose  $G$ is solvable.  
Then  $O_n(G)\ne 1$ for some $n\in \pi({G})$.
Since the length of $O_n(G)$-orbit on the vertices of  $\Theta$ divides $s(p+2),$
it follows that either $n=s$, or $n$ divides $p+2$.  

Further, $G$ contains a Hall $\{s,t\}$-subgroup $H$ for each $t\in \pi(p+2)$.
Let $g$ and  $f$ be some elements of orders $s$ and  $t$ of  $H$, respectively. 
 
Suppose $O_s(G)\ne 1$. We may assume that $\langle g\rangle=O_s(G)$. By Lemma~\ref{L5} 
    $f\notin C_G(g)$ and hence $|f|$ divides $s-1$. Since $(p^2+4p+1,p+2)$ divides 3,
   $|f|=3$. Thus, $p+2$ is a power of 3 and $s\equiv 1 \pmod 3$. 
Again by Lemma~\ref{L5} $\pi(C_X(g))\subseteq \pi(p+1)$ and since 
$X/C_X(g)\le Z_{s-1}$, we get  $\pi(X)\subseteq \pi((p+1)(p^2+4p+1))$. 
 
Now let $O_s(G)=1$. Then $G$ contains a minimal normal subgroup $N$ of order $t^e$
for some $t\in \pi(p+2)$. 
By Lemma~\ref{L5} $C_N(g)=1$. Hence
$t^e\equiv 1\pmod s, $ 
$\langle g\rangle\le GL_e(t)$ and $e\ge 2$. In particular, by Lemma~\ref{L3} the number   $p+2$ 
is composite.
The lemma is proved.
\end{proof}

\section{Automorphisms of ${\rm AT4}(p,p+2,r)$-graphs}\label{Sec2}

Further in this section we suppose that  $\Gamma$ is an ${\rm AT4}(p,p+2,r)$-graph, $\bar{\Gamma}$ is the antipodal quotient of  $\Gamma$, $G={\rm Aut}(\Gamma)$, 
   $F\in {\mathcal F}(\Gamma), a\in F$,  $\Phi=\Gamma_2(a)$,   $\Theta^i$ is the $i$-neighborhood of a vertex $F$ of   $\bar\Gamma$  and  ${\mathcal F}_i$ is that the set of those antipodal classes of $\Gamma$, which intersect   $\Gamma_i(a)$, where $i\in\{1,2\}$.
By $G_X$ ($G_{\{X\}}$) we denote pointwise (respectively, global) stabilizer of a subset  $X$ of vertices of  $\Gamma$ in $G$.

It is known (see~\cite{JKT,Jur}), that each  ${\rm AT4}(p,q,r)$-graph is 1-homogeneous in the sense of Nomura,
hence for each triple $\{x,y,z\}$ of vertices  of   $\Gamma$ such that
$d(x,y)=1$ and $d(x,z)=d(y,z)=2$,  we have $$|\Gamma_1(x)\cap \Gamma_1(y)\cap \Gamma_1(z)|=c_2(a_1-p)/a_2=2(p+1)/r.$$

 \begin{lemma}\label{L7} Let $K_i$ denote the kernel of the action of $G_{\{F\}}$ on ${\mathcal F}_i$,
 where $i\in \{1,2\}$. Then $K=K_1\cap K_2$
 is the kernel of the action of $G$ on ${\mathcal F}(\Gamma)$ and the following statements hold: 
 \begin{itemize}
 \item[$(1)$]   $G_a\cap K=1$ and $|K|$ divides $r$;
 \item[$(2)$] $G_{\Gamma_{i}(a)}=1$, $G_a\le {\rm Aut}(\Theta^i)$ for all $i\in \{1,2\}$ and $G_a\le {\rm Aut}(\Phi)$;
  \item[$(3)$] $K_1=K_2=K$ and $G_{\{F\}}/K\le {\rm Aut}(\bar{\Phi})$.
 \end{itemize} 
\end{lemma}
\begin{proof}
 Suppose that $1\ne g\in G_a\cap K$. Then $\Gamma_1(a)\subset{\rm Fix}(g)$. 
If there is a vertex $b\in \Gamma_2(a)-{\rm Fix}(g)$, then we  get $\Gamma_1(b)\cap \Gamma_1(b^g)\cap {\rm Fix}(g)\ne \varnothing$, which  means   $d({b},{b^g})\le 2$, a contradiction.  Hence   $\Gamma_2(a)\subset{\rm Fix}(g)$. In  a similar way, we obtain $\Gamma_3(a)\subset{\rm Fix}(g)$. But then $\Gamma={\rm Fix}(g)$, again a contradiction. 
 
Let $\tilde g\in {\rm Aut}(\bar{\Gamma})$ be an element of prime order and $\Omega={\rm Fix}(\tilde g)$.  
For a vertex $x$ of $\Gamma$ by  $\bar{x}$ we will denote a vertex of  $\bar{\Gamma}$, whose preimage in  $\Gamma$ (i.e. antipodal class) contains  $x$.   
Let us show that if  $\Omega\neq \varnothing$, then for each vertex $\bar{b}$ of $\bar{\Gamma}$  and for all $i\in\{1,2\}$ we have $\bar{\Gamma}_i(\bar{b})\nsubseteq \Omega$.

By Proposition~\ref{P1}, $\bar \Gamma$ is strongly regular with parameters 
  $$((p+1)^2(p+4)^2/2,(p+2)(p^2+4p+2),p(p+3),2(p+1)(p+2))$$ and non-principal eigenvalues 
  $p,-(p^2+4p+4)$.  
Let $\bar{a}\in \Omega\neq \varnothing$.  Then in view of \cite[Theorem 3.2]{BehLam} 
$$|\Omega|\le (p+1)(p+2)\frac{(p+1)^2(p+4)^2}{(p+2)(p^2+4p+2)-p}.$$  
By Proposition~\ref{P1},     $\bar{\Gamma}_2(\bar{a})$  is strongly regular with parameters 
 $$((p+1)(p+3)(p^2+4p+2)/2,p(p+2)^2,p^2+p-2,2p(p+1)).$$ 
Hence $\bar{\Gamma}_2(\bar{a})\nsubseteq \Omega$.   
Suppose that  $\bar{\Gamma}_1(\bar{a})\subseteq \Omega$. Then for a vertex
 $\bar{b}\in \bar{\Gamma}_2(\bar{a})-\Omega$ we have $|\bar{\Gamma}_1(\bar{b})\cap \bar{\Gamma}_1(\bar{b}^g)\cap \Omega|=2(p+1)(p+2)$.  But the vertices $\bar{b}$ and $\bar{b}^{\tilde g}$ are at distance at most  2 in  $\bar{\Gamma}_2(\bar{a})$, which is impossible.
    
Now suppose that $\bar{\Gamma}_i(\bar{b})\subseteq \Omega$ for a vertex $\bar{b}$ of $\bar{\Gamma}$ 
and some $i\in \{1,2\}$.
Since   $\bar{b}\notin \Omega$, the vertices  $\bar{b}$ and $\bar{b}^{\tilde g}$ have $|\bar{\Gamma}_1(\bar{b})|$  common neighbors in $\bar{\Gamma}$ or $|\bar{\Gamma}_2(\bar{b})|$ common neighbors in $\bar{\Gamma}_2$,  
a contradiction. 

Hence, $K_i=K$,
  $G_{\Gamma_i(a)}\unlhd G_a$ and  $G_{\Gamma_i(a)}$ fixes each antipodal class of $\Gamma$.
By statement $(1)$ it follows that  $G_{\Gamma_i(a)}=1$.
Therefore,  
$G_a\simeq G_aK/K\leq G_{\{F\}}/K\le {\rm Aut}(\Theta^i)$ and, similarly we get, $G_a\le {\rm Aut}(\Phi)$.
The lemma is proved.
\end{proof}

\begin{lemma}\label{L8} Let  $g\in G$ be an element of prime order and $\Omega={\rm Fix}(g)$. 
Then $|\Omega|\le{r(p+1)(p+2)(p+4)}$ and the following statements hold:

\begin{itemize}
\item[$(1)$] if $\Omega\neq \varnothing$, $p$ is a prime power and $p>2$,  then either
\begin{itemize}
\item[$(i)$]   $|g|\le p$, or
\item[$(ii)$]  $|g|=p+2$ and  $\Omega$ --- $2r$-coclique, which is a union of two antipodal classes of $\Gamma$,  
or
\item[$(iii)$]    $|g|>p, |g|$ divides $(p+2)^2-2$ and, in particular, if  $|g|=(p+2)^2-2$, then $\Omega$ is
an antipodal class of $\Gamma$;

\end{itemize}
\item[$(2)$] if $\Omega= \varnothing$, then $|g|$ divides $(p+1)(p+4)$.
 
  \end{itemize}
\end{lemma}
\begin{proof} Let $a\in \Omega\neq \varnothing$ and $\Theta=\Gamma_1(a)$.
If $g\notin K$ (which holds whenever $|g|$ does not divide $r$), then $g$
induces a non-trivial automorphism of the strongly regular graph $\bar{\Gamma}$ 
with parameters 
$$((p+1)^2(p+4)^2/2,(p+2)(p^2+4p+2),p(p+3),2(p+1)(p+2))$$ and non-principal eigenvalues 
  $p,-(p^2+4p+4)$, and by \cite[Theorem 3.2]{BehLam} 
  $$|\Omega|/r\le \frac{2(p+1)(p+2)\cdot v}{k-p}=
  \frac{(p+1)^3(p+2)(p+4)^2}{(p+1)^2(p+4)}={(p+1)(p+2)(p+4)}.$$
 We have $$x_1=|\Omega\cap \Gamma_1(a)|\equiv (p+2)(p^2+4p+2)\pmod {|g|},$$  
 $$x_2=|\Omega\cap \Gamma_2(a)|\equiv 
 \frac{(p^2+4p+2)(p+3)(p+1)r}{2}\pmod {|g|},$$  
 $$x_3=|\Omega\cap \Gamma_3(a)|\equiv (p+2)(p^2+4p+2)(r-1)\pmod {|g|},$$  
 $$x_4=|\Omega\cap \Gamma_4(a)|\equiv r-1\pmod {|g|},$$ 
 $$|\Gamma-\Omega|=v-1-x_1-x_2-x_3-x_4\equiv 0\pmod {|g|},$$
in particular, if $|g|$ does not divide $(p+1)(p+2)(p+3)(p^2+4p+2)r(r-1)$, then $x_1x_2x_3x_4>0$.

Clearly, if $|g|>r$, then the antipodal class containing a vertex $b$ from $\Omega$, is contained in $\Omega$ itself.
 
If $|g|>\max\{\lambda,\mu\}$, then by Proposition~\ref{P1} $|g|>r$ and, consequently, 
  $\Omega$ is an $r$-covering of a graph $\bar{\Omega}$. If moreover  $\Omega$ is connected,
  then $\Omega$ is an antipodal distance-regular graph of diameter 4 with intersection array 
$$\{x_1,(p+3)(p+1)^2,(r-1)2(p+1)(p+2)/r,1;
  1,2(p+1)(p+2)/r,(p+3)(p+1)^2,x_1\}$$ and by \cite[Corollary 4.2.5]{BCN} $p^2(p+3)^2+4x_1$ is a square.
In this case $\bar{\Omega}$ is a strongly regular graph with parameters $(v',x_1,p(p+3),2(p+1)(p+2))$.

Further we will assume that  $p>2$ and $p$ is a prime power. 

Let $|g|=(p+2)^2-2$. 
Then  $|g|>\max\{\lambda,\mu\}$.
Hence $x_3=(r-1)x_1$.
In view of Lemma~\ref{L7}   $g$ induces a non-trivial automorphism of ${\Theta}$, which together with Lemma~\ref{L3} implies $x_1=x_3=0$.
 Since $|g|>\max\{\lambda,\mu\}$,
 $x_2=0$, that is $\Omega$ coincides with the antipodal class $F$. 
 
If  $|g|$ does not divide $(p+2)^2-2$,  then by Lemmas~\ref{L3} and \ref{L7} it follows that either
 $|g|=p+2$, or $|g|\le p$.  

Let $|g|=p+2$. By Proposition~\ref{P1}, $|g|>r$ and $|g|$ does not divide $|\Phi|$.
Hence  $x_2>0$. By Lemma~\ref{L7},  $g$ induces a non-trivial automorphism of ${\Theta}$, which together with Lemma~\ref{L3} implies $x_1=0$.
By applying Lemma~\ref{L7} again, we obtain that $g$ 
induces a non-trivial automorphism of $\bar{\Phi}$, which is strongly
regular graph with parameters $$((p+1)(p+3)(p^2+4p+2)/2,p(p+2)^2,p^2+p-2,2p(p+1))$$
 and non-principal eigenvalues
  $p,-(p^2+2p+2)$. 
Let $x\in \Omega\cap \Phi$.  We have $$y_1=|\Omega\cap \Phi_1(x)|\equiv p(p+2)^2\pmod {|g|},$$  $$y_2=|\Omega\cap \Phi_2(x)|\equiv \frac{(p+2)^2(p+1)^2r}{2}\pmod {|g|},$$  $$y_3=|\Omega\cap \Phi_3(x)|\equiv p(p+2)^2(r-1)\pmod {|g|},$$ $$y_4=|\Omega\cap \Phi_4(x)|=r-1,$$ $$|\Phi-\Omega|=|\Phi|-r-y_1-y_2-y_3\equiv 0\pmod {|g|}.$$ 
If $y_1>0$ (which means that $g$ fixes an edge $\{x,y\}$ of $\Phi$), then by Lemmas~\ref{L3} and \ref{L7}
$x_1=x_3/(r-1)=0$ and the subgraph $\Gamma_1(x)\cap \Gamma_1(y)\cap \Gamma_1(a)$ is $\langle g\rangle$-invariant.
But  $|g|$ does not divide  $2(p+1)/r$, a contradiction.  
Suppose $y_2>0$. Then for a vertex $y\in \Phi_2(x)\cap \Omega$ the subgraph $\Phi_1(x)\cap \Phi_1(y)$ is $\langle g\rangle$-invariant. But  $|g|$ does not divide $2p(p+1)$, a contradiction.
Since $y_3=(r-1)y_1$, we get   $y_1=y_2=y_3=0$, that is $\Omega\cap \Phi$ is an antipodal class of $\Phi$. 
Hence $g$ fixes (pointwise) exactly two antipodal classes of $\Gamma$.
The lemma is proved.
\end{proof}

 \begin{proposition}\label{P2} Let 
$p$ be a prime power, $p>2$ and  $a\in \Gamma$. If the graph $\Gamma$ is arc-transitive and the numbers $(p+2)$ and $(p^2+4p+2)$ are primes, then   $G_a$ is an almost simple group.
\end{proposition}
\begin{proof}  Put $\Theta=\Gamma_1(a)$, $X=G_a$ and $\{s_1,s_2\}=\{p+2, p^2+4p+2\}$. 
Since $X$ is transitive on the vertices of $\Theta$ and $p+2$ is a prime, by Lemma~\ref{L6}   $X$ is non-solvable. 
Denote by $S(X)$ the solvable radical of  $X$. By the same argument,  $S(X)$ is intransitive on the vertices of $\Theta$.

Suppose $S(X)\ne 1$. 
It follows from Lemmas~\ref{L3} and  \ref{L7} that neither of the numbers $s_1^2$ or $s_1s_2$   divides $|S(X)|$.
Hence we may assume that every $S(X)$-orbit on the vertices of $\Theta$ has length $s_1$ and $S(X)$ 
contains an element  $f$ of order $s_1$.  Thus 
for each element $g$ of order $s_2$ of $X$ the group  $\langle f,g\rangle$ is solvable (e.g. see \cite[Theorem 1.1]{GKPS}) and hence it contains a Hall $\{s_1,s_2\}$-subgroup $Y$. Also, by Lemmas~\ref{L3} and \ref{L7}
we get that  $s_i^2$ cannot be a divisor of $|Y|$,
so that $|Y|=s_1s_2$.  Since a subgroup of order $p^2+4p+2$ must be normal in  $Y$ and $(p+2,p^2+4p+1)=1$,  $X$
contains an element of order $s_1s_2$, a contradiction to Lemma~\ref{L5}.

Therefore,  $S(X)=1$ and the socle $soc(X)$ of  $X$ is a direct product of some non-abelian simple  groups. 
Pick a minimal normal subgroup $N$ of $X$. 

Suppose that $N$ is transitive on the vertices of $\Theta$. Then $s_1s_2$ divides $|N|$ and since $X$ 
contains no element of order $s_1s_2$, we conclude that  $N$ is simple.
 
In case, when $N$ is intransitive on the vertices of $\Theta$, by Lemma~\ref{L7} we may assume that each $N$-orbit on the vertices of $\Theta$ has length $s_1$. Then again  $N$ is simple, because  $|X|$ is not divisible by $s_1^2$.

Now if there is a minimal normal subgroup $N_1$ of $X$,  $N_1\ne N$, then   $N_1$ centralizes $N$
and hence  $|X|$ is divisible by $s_i^2$ or $X$ contains an element of order  $s_1s_2$, a contradiction.
Therefore, $soc(X)$ is a non-abelian simple group, which implies that $X$ is almost simple.
This proves the proposition.
\end{proof}

Let us finish  {\it proofs of Theorem~\ref{T1} and Corollary~\ref{C1}.} 
Theorem~\ref{T1} follows immediately from Lemmas \ref{L7} and \ref{L8}. Let $p$ be a prime power and $p>2$. In view of Theorem~\ref{T1} together with Propositions~\ref{P1} and \ref{P2}, it suffices to note that if $\Gamma$ is arc-transitive, 
then  $|G:G_{a,b}|=(p+2)(p^2+4p+2)(p+1)^2(p+4)^2r/2$
for each vertex $b\in \Gamma_1(a)$. Corollary~\ref{C1} is proved.~$\hfill\square$
 \medskip

Finally, we provide a  {\it proof of Theorem~\ref{T2}.}
Suppose $\Gamma$ is an arc-transitive ${\rm AT4}(p,p+2,r)$-graph with $p\in \{3,5, 11,17,27\}$, $a\in \Gamma$
and $G={\rm Aut}(\Gamma)$.  Then  $p+2$ and $p^2+4p+2$ are primes, $p^2+4p+2\in \{23,47,167,359,839\}$  and by Proposition~\ref{P2}, $G_a$ is almost simple. 

If $p=3$, then $\pi(G_a)\subseteq\{2,3,5,23\}$, which contradicts to  \cite[Table 1]{Zav}. Here we note that the
conclusion of Theorem~\ref{T2} for  $p=3$ also  follows from an earlier result \cite[Corollary 2]{MN2}. 
The case $p=5$ was considered in \cite{Tsi19}.

Let $p>5$. Then $p^2+4p+2\ge 100$.

Suppose that $s=p^2+4p+2\in \pi(soc(G_a))$. Then by Lemma \ref{L8} and \cite[Table 2]{Zav},  $soc(G_a)$
is isomorphic to the one of the   groups  $A_s, A_{s+1},\ldots, A_{s'-1}$, where $s'$ is the least prime such that $s'>s$, or to $L_2(s)$. By Lemmas~\ref{L3} and \ref{L7} it follows that the ``alternating'' case is not possible.
Let $soc(G_a)\simeq L_2(s)$. Since $|L_2(s)|=s(s^2-1)/2$ and $(s^2-1,p+2)$ divides 3,  $p+2\in \pi({\rm Out}(L_2(s)))$, a contradiction.

Now suppose $s= p+2\in \pi(soc(G_a))$. Then  $s\in \{13,19,29\}$,  the length of each $soc(G_a)$-orbit on  $\Theta$ equals  $s$ and $p^2+4p+2\in \pi({\rm Out}(soc(G_a)))$. By applying Lemma~\ref{L8}, we get a contradiction with~\cite[Table 1]{Zav} and \cite{atlas}.
Theorem~\ref{T2} is proved.~$\hfill\square$

\end{document}